\documentclass{amsart}
\usepackage{amsmath}
\usepackage{amssymb}

\numberwithin{equation}{section}

\newtheorem{theorem}{Theorem}[section]
\newtheorem{proposition}[theorem]{Proposition}
\newtheorem{corollary}[theorem]{Corollary}
\newtheorem{lemma}[theorem]{Lemma}

\theoremstyle{definition}
\newtheorem{definition}[theorem]{Definition}

\newcommand{\al}{\alpha}

\newcommand{\clk}{\sigma}

\newcommand{\neva}{N}

\newcommand{\aream}{A}
\newcommand{\meal}{m_n}

\newcommand{\diin}{\Theta}

\newcommand{\er}{\varepsilon}
\newcommand{\za}{\zeta}

\newcommand{\ph}{\varphi}
\newcommand{\cph}{C_\varphi}
\newcommand{\hol}{\mathcal{H}ol}
\newcommand{\Dbb}{\mathbb D}
\newcommand{\Tbb}{\mathbb T}
\newcommand{\Rl}{\mathrm{Re}}

\newcommand{\tn}{\mathbb{T}^n}

\newcommand{\dn}{\mathbb{D}^n}

\newcommand{\knl}{k}
\newcommand{\kknl}{\mathbf{k}}
\newcommand{\Cbb}{\mathbb C}

\newcommand{\spr}{\Sigma}
\begin{document}

\title[Composition operators]{Composition operators between model and Hardy spaces}
\author{Evgueni Doubtsov}
\address{St.~Petersburg Department
of Steklov Mathematical Institute, Fontanka 27, St.~Petersburg 191023, Russia}
\email{dubtsov@pdmi.ras.ru}


\begin{abstract}
Let $n\ge 1$ and $\varphi: \mathbb{D}^n\to\mathbb{D}$ be a holomorphic function,
where $\mathbb{D}$ denotes the open unit disk of $\mathbb{C}$.
Let $\Theta: \mathbb{D} \to \mathbb{D}$ be an inner function and
$K^p_\Theta$, $p>0$, denote the corresponding model space.
We obtain characterizations of the compact composition operators
$C_\varphi: K^p_\Theta \to H^p(\mathbb{D}^n)$, $1<p<\infty$,
where $H^p(\mathbb{D}^n)$ denotes the Hardy space.
\end{abstract}

\maketitle

\section{Introduction}

Let $\Dbb = \{z\in \Cbb: |z|<1\}$ and $\Tbb = \partial \Dbb$.
Let $\meal$ denote the normalized Lebesgue measure on the torus $\tn$, $n\ge 1$.
Let $\hol(\dn)$ denote the space of holomorphic functions in the polydisk $\dn$, $n\ge 1$.
For $0<p<\infty$, the classical Hardy space $H^p=H^p(\dn)$ consists of
$f\in \hol(\dn)$ such that
\[
\|f\|_{H^p}^p = \sup_{0<r<1} \int_{\tn} |f(r\za)|^p\, d\meal(\za) < \infty.
\]
As usual, the Hardy space $H^p(\dn)$, $p>0$, is identified with the space
$H^p(\tn)$ of the corresponding boundary values.

\subsection{Composition operators}
Consider a holomorphic mapping $\ph: \dn\to\Dbb$, $n\ge 1$.
It is known that the composition operator $\cph: f\mapsto f\circ\ph$ maps $H^p(\Dbb)$
into $H^p(\dn)$, $p>0$. Indeed, assume that $f\in H^p(\Dbb)$.
The following argument is well known.
Let $h$ denote the Poisson integral of the boundary values $|f^\ast|^p$.
Then $|f\circ \ph|^p \le h\circ\ph$ and $h\circ\ph$ is a pluriharmonic function in $\dn$.
Hence, $f\circ\ph \in H^p(\dn)$,
as required.

Since $\cph$ maps $H^p(\Dbb)$ into $H^p(\dn)$, it is natural to ask about characterizations
of those symbols $\ph$ for which
$\cph: H^p(\Dbb)\to H^p(\dn)$ is a compact operator.
For $n=1$, an explicit approach based on the Nevanlinna counting function
was developed in \cite{ShJoel87}.
A different answer, in terms of the Clark measures, was obtained in  \cite{CM97}.

\subsection{Model spaces}
\begin{definition}\label{d_inner}
A holomorphic function
 $\diin:\Dbb \to \Dbb$ is called \textsl{inner} if
$|\diin(\za)|=1$ for $m_1$-a.e.\ $\za\in\Tbb$.
\end{definition}

As usual, in the above definition, $\diin(\za)$ denotes
$\lim_{r\to 1-} \diin(r\za)$.
It is known that the corresponding limit exists $m_1$-a.e.

For an inner function $\diin$ on $\Dbb$, the classical model space $K_\diin$ is defined as
\[
K_\diin = H^2(\Dbb)\ominus \diin H^2(\Dbb).
\]

\subsection{Composition operators on model spaces}
In the present work,
in terms of the Nevanlinna counting function and in terms of the Clark measures,
we characterize those symbols $\ph$ for which
$\cph: K_\diin \to H^2(\dn)$ is a compact operator.
For $n=1$, such characterizations were obtained in \cite{LM13}.
Also, in Theorem~\ref{t_comp_KTp1}, the analogous problem is solved for the operator
$\cph: K_\diin^p \to H^p(\Dbb^n)$, $p>1$, $n\ge 1$,
where $K_\diin^p := H^p \cap \diin \overline{H^p}$.
Observe that $K_\diin^2 = K_\diin$.

\subsection*{Organization of the paper}
Auxiliary results are collected in Section~\ref{s_aux}.
Compact composition operators $\cph: K_\diin \to H^2(\dn)$ are characterized in Section~\ref{s_KT}
with the help of the Nevanlinna counting function.
Real interpolation of Banach spaces is applied in Section~\ref{s_KT}
to prove that the compactness of the operator $\cph: K_\diin^p \to H^p(\dn)$, $n\ge 1$,
does not depend on the parameter $p$ for $1<p<\infty$.
For a one-component inner function $\diin$, a description of the compact composition operators $\cph: K_\diin \to H^2(\dn)$
in terms of the Clark measures is given in Section~\ref{s_clk}.

\section{Auxiliary results}\label{s_aux}
\subsection{Littlewood--Paley identity and its generalizations}\label{ss_HL}

For $f\in H^2(\Dbb)$, the Little\-wood--Paley identity states that
\begin{equation}\label{e_LP}
\|f\|^2_{H^2(\Dbb)} = |f(0)|^2 + 2\int_{\Dbb} |f^\prime(w)|^2 \log\frac{1}{|w|}\, d\aream(w),
\end{equation}
where $\aream$ denotes the area measure on the disk $\Dbb$.

\subsubsection*{Stanton's formula}
To study the composition operator generated by a holomorphic symbol $\phi:\Dbb \to \Dbb$,
J.~H.~Shapiro \cite{ShJoel87} used for $f\circ\phi$ an analog of the identity \eqref{e_LP}.
This analog is based on the Nevanlinna counting function $\neva_\phi$ defined by
\[
\neva_{\phi}(w) = \sum_{z\in \Dbb:\, \phi(z)=w} \log\frac{1}{|z|}, \quad w\in\Dbb\setminus\{\phi(0)\},
\]
where each pre-image is counted according to its multiplicity.
The following Stanton formula is the principal technical tool in Shapiro's argument.

\begin{theorem}[\cite{ShJoel87}]\label{t_Stn_disk}
Let $\phi: \Dbb\to\Dbb$ be a holomorphic function.
Then
\begin{equation}\label{e_Stntn}
\|f\circ \phi\|^2_{H^2(\Dbb)} = |f(\phi(0))|^2 + 2\int_{\Dbb} |f^\prime(w)|^2
\neva_{\phi}(w) \, d\aream(w).
\end{equation}
\end{theorem}

For a function $f\in \hol(\dn)$ and a point $\za\in\tn$, the slice-function
$f_\za\in\hol(\Dbb)$ is defined by the equality $f_\za(\lambda) = f(\lambda\za)$, $\lambda\in\Dbb$.

\begin{corollary}\label{c_Stanton}
Let $\ph: \dn\to\Dbb$, $n\ge 1$, be  a holomorphic function.
Then
\begin{equation}\label{e_Stntn_d}
\|f\circ \ph\|^2_{H^2(\dn)} = |f(\ph(0))|^2 + 2\int_{\Dbb} |f^\prime(w)|^2
\left(\int_{\tn} \neva_{\ph_\za}(w)\,  d\meal(\za) \right)\, d\aream(w).
\end{equation}
\end{corollary}
\begin{proof}
Let $\za\in \tn$.
Applying Theorem~\ref{t_Stn_disk} with $\phi = \ph_\za$ and integrating with respect to  the normalized measure
$\meal$ on $\tn$, we obtain the required equality \eqref{e_Stntn_d}.
\end{proof}

\subsection{Subharmonic property for the Nevanlinna counting function}

\begin{proposition}[{\cite[Section 4.6]{ShJoel87}}]\label{p_subharm}
Let $w\in\Dbb$ and $\phi: \Dbb\to\Dbb$ be a holomorphic function.
Let $\Delta$ be a disk centered at $w$ and such that $\phi(0)\notin \Delta$.
Then
\begin{equation}\label{e_subharm_1}
\neva_{\phi}(w) \le \frac{1}{\aream(\Delta)}\int_{\Delta} \neva_\phi(z)\, d\aream(z).
\end{equation}
\end{proposition}

\begin{corollary}\label{c_subharm}
Let $w\in\Dbb$ and $\ph: \dn\to\Dbb$ be a holomorphic function.
Let $\Delta$ be a disk centered at $w$ and such that $\ph(0)\notin \Delta$.
Then
\begin{equation}\label{e_subharm}
\int_{\tn} \neva_{\ph_\za}(w)\,d\meal(\za) \le \frac{1}{\aream(\Delta)}\int_{\Delta}
\left(\int_{\tn} \neva_{\ph_\za}(z)\,  d\meal(\za) \right)\, d\aream(z).
\end{equation}
\end{corollary}
\begin{proof}
Let $\za\in \tn$.
Using Proposition~\ref{p_subharm} with $\phi = \ph_\za$,
integrating with respect to $\meal$ and applying Fubini's theorem, we obtain the required
inequality \eqref{e_subharm}.
\end{proof}

\subsection{Reproducing kernels for $K_\diin$}

Recall that the reproducing kernel $\knl_w(z)$ for $K_\diin$ is defined by the following equality:
\[
\knl_w(z) = \frac{1-\diin(z)\overline{\diin}(w)}{1- z\overline{w}}, \quad
\|\knl_w\|^2 = \frac{1-|\diin(w)|^2}{1-|w|^2}.
\]

\begin{lemma}[{\cite[Lemma~1]{LM13}}]\label{l_LM}
Let $\{w_q\}_{q=1}^\infty \subset\Dbb$ be a sequence such that $|w_q|\to 1$ as $q\to \infty$ and
\begin{equation}\label{e_LM}
|\diin (w_q)| < a
\end{equation}
for some parameter $a\in (0,1)$. Then
$\knl_{w_q}/\|\knl_{w_q}\| \overset{w^\ast}\longrightarrow 0$ as $q\to \infty$.
\end{lemma}

\section{Compact composition operators on model spaces}\label{s_KT}

\subsection{Compact composition operators on $K_\diin$}

\begin{theorem}\label{t_cph_tri}
Let $\ph: \dn\to\Dbb$, $n\ge 1$, be a holomorphic function and $\diin: \Dbb \to \Dbb$ be an inner function
such that $K_\diin$ is infinite dimensional.
Then the following properties are equivalent.
\begin{itemize}
  \item [(i)] One has
  \begin{equation}\label{e_nevaTo0}
\int_{\tn}\neva_{\ph_\za}(w) \frac{1-|\diin(w)|}{1-|w|}\, d\meal(\za)\to 0 \quad \textrm{as}\ |w|\to 1-.
\end{equation}
  \item [(ii)] The operator $\cph: K_\diin \to H^2(\dn)$ is compact.
\end{itemize}
\end{theorem}
\begin{proof}[About the proof]
The principal argument is adaptable from \cite{ShJoel87}.
In fact, the proof essentially coincides with that for $\ph$ defined on the unit ball of $\mathbb{C}^n$;
see \cite{Dou25}.
\end{proof}

\subsection{Compact composition operators on $K_\diin^p$, $p>1$}\label{s_KTp}

For $0<p<\infty$ and an inner function $\diin$, put
\[
K^p_\diin = K^p_\diin(\Dbb) \overset{\mathrm{def}}= H^p(\Dbb) \cap \diin \overline{H^p}(\Dbb).
\]
It is well known and easy to see that $K^2_\diin = K_\diin$.

By definition, an inner function $\diin: \Dbb \to \Dbb$ is called \textit{one-component} if the set
$\{z\in\Dbb: |\diin(z)|< r\}$ is connected for some parameter $r\in (0,1)$.
This section is motivated by the following assertion.

\begin{proposition}[{\cite[Section~4]{LM13}}]\label{p_BaLM}
Let $\phi: \Dbb\to\Dbb$ be a holomorphic function, $p>1$ and $\diin$ be a one-component inner function.
Then the operator $C_\phi: K^p_\diin \to H^p(\Dbb)$ is compact if and only if
\[
\neva_{\phi}(w) \frac{1-|\diin(w)|}{1-|w|} \to 0 \quad \textrm{as}\ |w|\to 1-.
\]
\end{proposition}

In Theorem~\ref{t_comp_KTp1} below, we show that
a direct analog of Proposition~\ref{p_BaLM}
holds for an arbitrary inner function $\diin$.
We use the real interpolation method for Banach spaces,
thus, first we recall the corresponding basic facts.

Let $(A_0, A_1)$ be a compatible pair of Banach spaces.
For parameters $0< \theta < 1$ and  $1\le q \le \infty$, the real method of interpolation
generates $(A_0, A_1)_{\theta, q}$, an interpolation space between $A_0$ and $A_1$
(see, for example, \cite[Chapter~3]{BL76} for details).

We need the following one-sided compactness theorem for the method of real interpolation.

\begin{theorem}[\cite{Cw92}]\label{t_cmpIntrp_Cw}
Let $(A_0, A_1)$ and $(B_0, B_1)$ be compatible pairs of Banach spaces.
Assume that the linear operators $T: A_j \to B_j$, $j=0, 1$, are bounded and
$T:A_0 \to B_0$ is a compact operator.
Then $T: (A_0, A_1)_{\theta, q}  \to (B_0, B_1)_{\theta, q}$ is a compact operator for all admissible
parameters $\theta$ and $q$.
\end{theorem}

\begin{theorem}\label{t_comp_KTp1}
Let $1<p<\infty$ and $\diin$ be an inner function such that $K_\diin$ is infinite dimensional.
Then $\cph: K^p_\diin \to H^p(\dn)$, $n\ge 1$, is a compact operator if and only if
property \eqref{e_nevaTo0} holds.
\end{theorem}

\begin{proof}
By Theorem~\ref{t_cph_tri}, it suffices to prove that the compactness of the opera\-tor
$\cph: K^p_\diin(\Dbb)\to H^p(\dn)$, $n\ge 1$, does not depend on the parameter $p\in (1, \infty)$.
To prove this property, assume that $p_0\in (1, \infty)$
and $\cph: K^{p_0}_\diin(\Dbb)\to H^{p_0}(\dn)$, $n\ge 1$, is a compact operator.

Fix $p$ and $p_1$ such that $p_1 > p > p_0$ or $1<p_1 < p < p_0$.
Define $\theta\in (0,1)$ by the following equality:
\begin{equation}\label{e_theta}
\frac{1}{p} = \frac{1-\theta}{p_0} + \frac{\theta}{p_1}.
\end{equation}

On the one hand, equality
\eqref{e_theta} guarantees that
\[
(H^{p_0}(\dn), H^{p_1}(\dn))_{\theta, p} = H^{p}(\dn), \quad n\ge 1.
\]
Indeed, this interpolation formula follows from the corresponding result
for the couple $(L^{p_0}(\tn), L^{p_1}(\tn))$, $1< p_0, p_1 < \infty$, by
applying the Riesz projection.
See\ \cite{Ki99} for further results on interpolation of Hardy spaces.

On the other hand, \eqref{e_theta} gives
\[
(K_\diin^{p_0}, K_\diin^{p_1})_{\theta, p} = K_\diin^{p}.
\]
This interpolation formula is known and
follows from the corresponding result
for the couple $(L^{p_0}(\Tbb), L^{p_1}(\Tbb))$, $1< p_0, p_1 < \infty$, by
applying the projection $P_\diin = P- \diin P \overline{\diin}$,
where $P$ denotes the Riesz projection.

Now, observe that the operator $\cph: K^{p_1}_\diin(\Dbb)\to H^{p_1}(\dn)$, $n\ge 1$, is bounded.
Indeed, as indicated in the introduction, $C_\ph: H^t(\Dbb) \to H^t(\dn)$ is a bounded operator for all $0<t<\infty$.
Thus, applying Theorem~\ref{t_cmpIntrp_Cw} with $q=p$, we conclude that the operator
$\cph: K^{p}_\diin(\Dbb)\to H^{p}(\dn)$, $n\ge 1$, is compact, as required.
\end{proof}

\section{One-component inner functions and Clark measures}\label{s_clk}

\subsection{Spectrum of a one-component inner function}
Given an inner function $\diin$, consider its canonical factorization
\[
\diin(z) = B_\Lambda \exp \left( \int_{\Tbb} \frac{z+\za}{z-\za}\, d\mu(\za)\right ), \quad z\in \Dbb,
\]
where $\Lambda$ is the zero set of $\diin$, $B_\Lambda$ is the corresponding Blaschke product,
 $\mu$ is a positive singular measure. The spectrum $\spr (\diin)$ is defined as
 \[
 \spr(\diin) = (\Tbb \cap \mathrm{clos\,} \Lambda) \cup \mathrm{supp\,}\mu.
 \]

The following lemma provides characterizations of the spectrum for
a one-compo\-nent inner function.

\begin{lemma}[{\cite[Section~5]{VT88}}]\label{l_A}
Let $\diin$ be a one-component inner function and $\al\in\Tbb$. Then the following properties are equivalent.
\begin{enumerate}
  \item[$\bullet$] $\al\in\spr(\diin)$;
  \item[$\bullet$] $\liminf_{w\to\al} |\diin(w)| <1$;
  \item[$\bullet$] $\liminf_{r\to 1-} |\diin(r\al)| <1$.
\end{enumerate}
\end{lemma}

In this section, we apply Clark measures to describe the compact operators $\cph$.

\subsection{Clark measures}
Given an $\alpha\in\Tbb$ and a holomorphic function $\ph: \dn\to \Dbb$, the quotient
\[
\frac{1-|\ph(z)|^2}{|\al-\ph(z)|^2}= \Rl \left(\frac{\al+ \ph(z)}{\al- \ph(z)} \right), \quad z\in \dn,
\]
is positive and pluriharmonic.
Therefore, there exists a unique positive measure $\clk_\al= \clk_\al[\ph]$ on $\tn$
such that
\[
P[\clk_\al](z) = \Rl \left(\frac{\al+ \ph(z)}{\al- \ph(z)} \right), \quad z\in \dn,
\]
where $P[\clk_\al]$ denotes the Poisson integral of $\clk_\al$, that is,
\[
P[\clk_\al](z) = \int_{\tn} \prod_{j=1}^{n} \frac{1-|z_j|^2}{|1-z_j \overline{\za}_j|^2}\, d\clk_\al(\za), \quad z\in \dn.
\]
By definition, $\clk_\al$ is called a Clark measure.

Let $\clk_\al = \clk_\al^a + \clk_\al^s$ be the Lebesgue decomposition
of the measure $\clk_\al$ with respect to $\meal$.
Observe that
\[
\|\clk_\al\| = P[\clk_\al](0) = \frac{1-|\ph(0)|^2}{|\al - \ph(0)|^2}
\]
and
\[
d\clk_\al^a (\za)= \frac{1-|\ph(\za)|^2}{|\al - \ph(\za)|^2}\, d\meal(\za).
\]
See \cite{AD20} 
for further results about Clark measures in several variables.

\subsection{Compact composition operators and Clark measures}

\begin{theorem}\label{T_2}
Let $\diin$ be a one-component inner function such that $K_\diin$ is infinite dimensional
and let $\ph: \dn\to \Dbb$ be a holomorphic function, $n\ge 1$.
Then $\cph: K_\diin \to H^2(\dn)$ is a compact operator if and only if
\begin{equation}\label{e_clark}
\|\clk_\al^s\| =0 \quad \textrm{for all $\al\in\spr(\diin)$}.
\end{equation}
\end{theorem}
\begin{proof}
Assume that $\cph$ is a compact operator.
Consider a point $\al\in\spr(\diin)$.
By Lemma~\ref{l_A}, there exists a sequence $\{r_q\}$ such that $r_q\nearrow 1-$ and
\begin{equation}\label{e_less1}
  \lim_{q\to \infty} |\diin(r_q\al)| <1.
\end{equation}

Lemma~\ref{l_LM} guarantees that
\[
\int_{\tn}\frac{|1-\diin(\ph(\za))\overline{\diin}(r_q\al)|^2}{|1- r_q\ph(\za)\overline{\al}|^2} \frac{1-|r_q|^2}{1-|\diin(r_q\al)|^2}\,d\meal(\za)
= \frac{\|\cph \knl_{r_q\al}\|^2_{2}}{\|\knl_{r_q\al}\|^2_2} \to 0\quad \textrm{as}\ q\to\infty.
\]
By \eqref{e_less1}, we obtain
\begin{equation}\label{e_null}
  \int_{\tn}\frac{1-|r_q|^2}{|1- r_q\ph(\za)\overline{\al}|^2}
\to 0\quad \textrm{as}\ q\to\infty.
\end{equation}
Now, we argue as in \cite{CM97} for $n=1$.
One has
\[
\frac{1-|r_q|^2}{|1- r_q\ph(\za)\overline{\al}|^2}
= \frac{1-|r_q \ph(\za)|^2}{|1- r_q\ph(\za)\overline{\al}|^2}
-  |r_q|^2 \frac{1-|\ph(\za)|^2}{|1- r_q\ph(\za)\overline{\al}|^2}.
\]
Observe that the pluriharmonic function
\[
\frac{1-|r_q \ph(\za)|^2}{|1- r_q\ph(\za)\overline{\al}|^2}
\]
is bounded on the polydisk $\dn$,
thus,
\[
\int_{\tn}\frac{1-|r_q \ph(\za)|^2}{|1- r_q\ph(\za)\overline{\al}|^2}\, d\meal(\za) =
\frac{1-|r_q \ph(0)|^2}{|1- r_q\ph(0)\overline{\al}|^2}\overset{q\to \infty}\longrightarrow  \|\clk_\al\|.
\]
On the other hand, by the monotone convergence theorem,
\[
\int_{\tn}|r_q|^2 \frac{1-|\ph(\za)|^2}{|1- \ph(\za) r_q\overline{\al}|^2}\, d\meal(\za) \overset{q\to \infty}\longrightarrow
\int_{\tn}\frac{1-|\ph(\za)|^2}{|1- \ph(\za)\overline{\al}|^2}\, d\meal(\za) = \|\clk_\al^a\|.
\]
Applying \eqref{e_null}, we obtain
\[
\|\clk_\al^s\| = \|\clk_\al\| - \|\clk_\al^a\| =0,
\]
as required.

To prove the reverse implication, assume that condition \eqref{e_clark} is satisfied.
By Theo\-rem~\ref{t_cph_tri}, it suffices to verify property \eqref{e_nevaTo0}.
Suppose that \eqref{e_nevaTo0} does not hold.
Then there exist a sequence $\{w_q\}_{q=1}^\infty$ and a point $\al\in\Tbb$ such that
$w_q\to \al$ as $q\to \infty$ and
\[
\int_{\tn}\neva_{\ph_\za}(w_q) \frac{1-|\diin(w_q)|}{1-|w_q|}\, d\meal(\za) > c_0 > 0.
\]
By Littlewood's inequality,
\[
\neva_{\ph_\za}(w) \le C (1-|w|), \quad\za\in\tn,\ \frac{1+|\ph(0)|}{2} < |w| <1.
\]
Therefore, property \eqref{e_LM} holds. Also, $\al\in\spr(\diin)$ by Lemma~\ref{l_A} and
\begin{equation}\label{e_NgC}
\int_{\tn} \frac{\neva_{\ph_\za}(w_q)}{1-|w_q|}\, d\meal(\za) > c > 0
\end{equation}
for all sufficiently large $q$.

Now, we consider the bounded operator $\cph: H^2(\Dbb) \to H^2(\dn)$.
Recall that the reproducing kernel for $H^2(\Dbb)$ is defined by
\[
\kknl_w (z) = \frac{1}{1- z\overline{w}}, \quad \|\kknl_w\|_2^2 = \frac{1}{1-|w|^2}.
\]
Let $D_\er(w) = \{z\in\Dbb: |z-w| < \er|1- z\overline{w}|\}$, that is,
let $D_\er(w)$ denote the pseudohyperbolic $\er$-disk centered at  $w\in \Dbb$.
Observe that
\begin{equation}\label{e_kknl}
|\kknl_w^\prime (z)| \ge C_\er \frac{|w_q|}{(1-|w_q|)^2}, \quad z\in D_\er(w).
\end{equation}
Sequentially applying equality \eqref{e_Stntn_d}, estimate~\eqref{e_kknl} and Corollary~\ref{c_subharm}, we obtain
\begin{equation}\label{e_rker_estim}
\begin{aligned}
\frac{\|C_\ph \kknl_{w_q}(z)\|^2}{\|\kknl_{w_q}\|^2}
    &\ge \frac{1}{\|\kknl_{w_q}\|^2} \int_{\Dbb} |\kknl^\prime_{w_q}(z)|^2
\left(\int_{\tn} \neva_{\ph_\za}(z)\,  d\meal(\za) \right)\, d\aream(z)\\
    &\ge \int_{D_\er(w_q)} \frac{C}{(1- |w_q|^2)^{3}}
\left(\int_{\tn} \neva_{\ph_\za}(z)\,  d\meal(\za) \right)\, d\aream(z)\\
    &\ge \frac{C_\er}{1- |w_q|^2} \int_{\tn} \neva_{\ph_\za}(w_q)\,d\meal(\za).
\end{aligned}
\end{equation}

On the one hand, using \eqref{e_rker_estim} and \eqref{e_NgC}, we conclude that
\[
\frac{\|\cph \kknl_{w_q}\|_2^2}{\|\kknl_{w_q}\|_2^2}
\ge C \int_{\tn} \frac{\neva_{\ph_\za}(w_q)}{1-|w_q|}\, d\meal(\za) > C\cdot c > 0.
\]
On the other hand, applying Fatou's lemma and condition \eqref{e_clark}, we obtain
\[
\begin{split}
   &\limsup_{q\to \infty} \frac{\|\cph \kknl_{w_q}\|_2^2}{\|\kknl_{w_q}\|_2^2}
   \\
&\le \limsup_{q\to \infty} \int_{\tn} \frac{1-|w_q\ph(\za)|^2}{|1- \overline{w_q}\ph(\za)|^2} \, d\meal(\za)
- \liminf_{q\to \infty} \int_{\tn} |w_q|^2 \frac{1-|\ph(\za)|^2}{|1- \overline{w_q}\ph(\za) |^2}\, d\meal(\za) \\
&\le \frac{1-|\ph(0)|^2}{|1- \overline{\al}\ph(0)|^2}
- \int_{\tn} \frac{1-|\ph(\za)|^2}{|1- \overline{\al}\ph(\za) |^2}\, d\meal(\za) \\
     &= \|\clk_\al\| - \|\clk_\al^a\| = \|\clk_\al^s\| =0,
\end{split}
\]
a contradiction. Therefore, the required property \eqref{e_nevaTo0} holds.
\end{proof}

\bibliographystyle{amsplain}

\end{document}